\documentclass[letterpaper, 10 pt, conference]{ieeeconf}  

\IEEEoverridecommandlockouts                              
\usepackage{latexsym} 
\usepackage{amsmath} 
\usepackage{amsfonts}
\usepackage{amssymb}

\newcommand{\R}{\mathbb R}

\newtheorem{theorem}{Theorem}[section]
\newtheorem{prop}[theorem]{Proposition}
\newtheorem{lemma}[theorem]{Lemma}

\newtheorem{rem}[theorem]{Remark}
\newtheorem{example}[theorem]{Example}

\title{\LARGE \bf
Some results on second order controllability conditions
}

\author{Pierpaolo Soravia
\thanks{Dipartimento di Matematica, via Trieste 63, Universit\`a di Padova, 35121 Padova, Italy
        {\tt\small soravia@math.unipd.it}}%
}

\begin{document}

\maketitle
\thispagestyle{empty}
\pagestyle{empty}

\begin{abstract}

For a symmetric system, we want to study the problem of crossing an hypersurface in the neighborhood of a given point, when we suppose that all of the available vector fields are tangent to the hypersurface at the point. Classically one requires transversality of at least one Lie bracket generated by two available vector fields. However such condition does not take into account neither the geometry of the hypersurface nor the practical fact that in order to realize the direction of a Lie bracket one needs three switches among the vector fields in a short time. We find a new sufficient condition that requires a symmetric matrix to have a negative eigenvalue. This sufficient condition, which contains either the case of a transversal Lie bracket and the case of a favorable geometry of the hypersurface, is thus weaker than the classical one and easy to check. Moreover it is constructive since it provides the controls for the vector fields to be used and produces a trajectory with at most one switch to reach the goal.
\end{abstract}

\section{Introduction}

We consider the following controlled dynamical system
\begin{equation}\label{eqsys}
\left\{\begin{array}{ll}\dot x_t=\sigma(x_t)a_t,\\
x_0\in {\mathbb R}^n,
\end{array}\right.
\end{equation}
where $\sigma:\R^n\to\R^{n\times m}$ is locally Lipschitz continuous and $a.:[0,+\infty)\to B_1(0)$, $B_1(0)=\{a\in \R^m:|a|\leq1\}$, will always be piecewise continuous. For convenience we will indicate as $\sigma_i:\R^n\to\R^n$, $i=1,\dots,m$ the vector fields provided by the columns of the matrix valued $\sigma$. We will assume  at least $\sigma\in C^1(\R^n;\R^m)$.
Such a system is said to be symmetric.

We are also given a function $u:\R^n\to\R$, $u\in C^2(\R^n)$ and a point $\bar x\neq0$, $\nabla u(\bar x)\neq 0$.  The idea is that $u$ describes locally around $\bar x$ the boundary of a target set for (\ref{eqsys}), ${\bf T}=\{x:u(x)\leq u(\bar x)\}$.
The problem that we want to investigate is the following: find appropriate conditions on $\sigma$ so that for all positive and small times $t>0$ there is a piecewise constant control $a_.$ and some $\delta>0$ such that the corresponting trajectory of (\ref{eqsys}) starting out at every point $x_0$, $|x_0-\bar x|\leq\delta$, will satilsfy
$$u(x_t)\leq u(\bar x)-\delta.$$
This implies that the target ${\bf T}$ is small time locally attainable (STLA) by system (\ref{eqsys}) at $\bar x$.
In other words, given any small time $t>0$, the point $\bar x$ is in the interior of the set of points from which we can reach $\bf T$ in time less than $t$. Here function $u$ may for instance be the signed distance function from the manifold $M=\{x:u=u(\bar x)\}$ but this is not required.
It is well known, see the book \cite{bcd}, that (\ref{eqsys}) being STLA is equivalent to the continuity of the minimum time function at $\bar x$. Moreover, when (\ref{eqsys}) is STLA at every point in the boundary of the target, then the minimum time function is continuous in its whole domain. This fact makes it possible to characterise the minimum time function as the unique solution of a free boundary problem for the Hamilton-Jacobi equation, see Bardi and the author \cite{baso2}.

Sufficient conditions for system (\ref{eqsys}) to be STLA are given on the vector fields $\sigma_i$ and have different nature. Classical first order attainability conditions require that at least one of the vector fields is transversal to $M$, namely $\nabla u\cdot \sigma_i(\bar x)\neq0$ for some $i=1,\dots,m$ (Petrov condition). If this sufficient condition fails, namely all vector fields $\sigma_i(\bar x)$ are in the tangent space of $M$ at $\bar x$, one can give second order conditions. Classically these require that a Lie bracket between two vector fields of the system is transversal to $M$ at $\bar x$. There are at least two unpleasant facts in such a sufficient condition. The first one is that, as we see in Example 1 below, sometimes we can cross the manifold $M$ with only one vector field if the geometry of the manifold (in particular its curvature) is favorable as compared to the trajectory and still have a second order STLA. The second and more important for applications is that in order to construct a trajectory that follows (with an error) the new vector field provided by a Lie bracket, one needs to build a trajectory
that has three switches, and they need to happen in short time, if we want to keep the error small. Indeed in order for the system to follow the vector field $[\sigma_i,\sigma_j]$, which denotes the Lie bracket between the two vector fields, for small time $t>0$ we need to use the control
$$a_s=\left\{
\begin{array}{ll}
e_i,\quad& s\in[0,t)   \\
e_j,\quad& s\in[t,2t)   \\  
-e_i,\quad& s\in[2t,3t)   \\  
-e_j,\quad& s\in[3t,4t),
\end{array}
\right.$$
where $e_i,e_j$ are the usual elements of the canonical basis in $\R^m$. Time needs to be short because we still get an error of size $t^3$ when $\sigma \in C^2$. 

In this paper we derive explicit conditions on the vector fields of the system that imply appropriate estimates on the trajectories. Such estimates ensure STLA at $\bar x$ and then with a standard mechanism, continuity and then local 1/2-H\"older regularity of the minimum time function, see e.g. \cite{so}, or Theorem IV.1.18 in (\cite{bcd}). 
To do this  we consider trajectories of (\ref{eqsys}) with only one switch between two (at most) vector fields. When the boundary of the target is locally a smooth manifold of codimension 1, this family of trajectories is enough to provide second order local attainability at one point. Our conditions can be easily described by checking if a suitable positive semidefinite matrix has a negative eigenvalue, and they contain either the case of a transversal Lie bracket and the case of a single vector field with {\it good} curvature. Therefore our sufficient condition is more general than the classical one. Moreover, the corresponding eigenvector contains the coordinates of the two controls that we can use to define a trajectory crossing the hypersurface. Our method is therefore constructive in that finding an eigenvector with negative eigenvalue will deliver what controls to use, and optimality, since the least eigenvector corresponds to $u$ decreasing with the highest rate. The main difference in our approach from others in the literature is that instead of looking at sufficient conditions for second order controllability as $\nabla u\cdot [\sigma_i,\sigma_j](x)\neq0$ for some $i,j$, which is a first order operator on the function $u$, we find it natural to look at properties of the second order hamiltonians $\nabla(\nabla u\cdot\sigma_i)\cdot\sigma_j$ which are second order operators in $u$. As we state below, our approach can then be rephrased as some second order, fully nonlinear elliptic partial differential equation having a smooth strict supersolution making it a counterpart for second order conditions of the Hamilton-Jacobi equation in the case of Petrov condition. Our work can have consequences on the way one can construct controls that steer even globally a system to a target, in finite time or asymptotically, but we will not discuss it here.
Higher order attainability conditions are also possible, but we will not study them in this paper.

Small time local attainability and regularity of the minimum time function is a long studied and important subject in optimal control.
Besides classical results by Kalman (for linear systems) and by Sussman, who mainly deal with controllability at equilibrium points of the system, we recall Petrov \cite{pe,pe2} for the study of first order controllability, that is attainability at a single point. Liverowskii \cite{li} studied the corresponding problem of second order, see also Bianchini and Stefani \cite{bs,bs2}. Controllability of higher order to a point was studied by Liverowskii \cite{li2}. For attainability of a target different from a point we recall the papers by Bacciotti \cite{ba} in the case of targets of codimension 1 and the author \cite{so} for manifolds of any dimension and possibly with a boundary.

More recently the work by Krastanov and Quincampoix \cite{kr,kr2} pointed out the importance of the geometry of the target and studied higher order attainability of nonsmooth targets for affine systems with nontrivial drift.  For the same class of systems Marigonda, Rigo and Le \cite{ma,ma2,ma3} studied higher order regularity focusing on the lack of smoothness of the target and the presence of state constraints. We finally mention the paper by Motta and Rampazzo \cite{mr} where the authors study higher order hamiltonians obtained by adding iterated Lie brackets as additional vector fields, in order to prove global asymptotic controllability to a target.

As a general notation, in the following we indicate the scalar product as
$$a\cdot b,\quad a,b\in\R^n,$$
and as ${}^tA$ the transpose of a matrix $A$. If $f:\R^n\to\R^n$ is a smooth vector field, we will denote $Df(x)=\left(\partial_{x_j}f_i\right)_{i,j=1\dots,n}$ the jacobian matrix of $f$ at $x$. As a general rule, in the product of functions having the same dependance, we will show their argument only after the last factor. Everything we develop will be in an appropriate neighborhood of a given point, but we keep all functions everywhere defined for convenience.

\section{Preliminaries}

In this section we recover the main estimates that we are going to use.

\subsection{Estimates of trajectories}

We first analyze the trajectory resulting from a unique switch between two smooth (at least of class $C^1$) vector fields.
For $f,g:\R^n\to \R^n$ vector fields of class $C^1$, we will use below the notation of Lie bracket as the vector field
$$[f,g]:\R^n\to\R^n,\quad [f,g](x)=Dg\;f(x)-Df\;g(x).$$

\begin{prop}\label{prop1}
Let $t>0$ and $f,g:\R^n\to\R^n$ be $C^1$ vector fields. 
Consider the solution of
\begin{equation}\label{eqswitch}
\dot x_s=\left\{\begin{array}{ll}f(x_s),\quad&\mbox{ if }s\in[0,t),\\
g(x_s),&\mbox{ if }s\in(t,2t],\\ x_0\in\R^n.\end{array}\right.
\end{equation}
Then, as $t\to0+$,
$$\begin{array}{l}
x_{2t}-x_0\\
=(f(x_0)+g(x_0))t+D(f+g)\;(f+g)(x_{0})\frac{t^2}2\\
+[f,g](x_0)\frac{t^2}2+o(t^2).
\end{array}$$
\end{prop}

\begin{proof} Observe that
$$\begin{array}{ll}
x_{2t}-x_0=x_{2t}-x_t+x_t-x_0\\
=\int_0^t(g(x_{t+s})+f(x_s))\;ds\\
=\int_0^t[(g(x_{t+s})-g(x_t))+(f(x_s)-f(x_0)]\;ds\\
+g(x_t)t+f(x_0)t\\
=\int_0^t[\int_0^s(Dg\;g(x_{t+r})+Df\;f(x_r))\;dr]\;ds\\
+\int_0^t(g(x_s)+sDg\;f(x_s))\;ds+f(x_0)t\\
=\int_0^t[\int_0^s(Dg\;g(x_{t+r})+Df\;f(x_r)+Dg\;f(x_r))\;dr]\;ds\\
+\int_0^tsDg\;f(x_s)\;ds+(f(x_0)+g(x_0))t\\
=(Dg\;g(x_{0})+Df\;f(x_0)+2Dg\;f(x_0))\frac{t^2}2\\
+(f(x_0)+g(x_0))t+o(t^2).
\end{array}
$$
\end{proof}

\begin{rem} Note that standard Taylor estimate shows that if we consider the averaged system
$$\begin{array}{ll}
\dot y_s=\frac{f(y_s)+g(y_s)}2,\quad&\mbox{ for }s\in[0,2t],\\
y_0=x_0,
\end{array}$$
which is a feasible trajectory of (\ref{eqsys}) by convexity of the control set, then it satisfies
$$\begin{array}{l}y_{2t}-x_0\\
=(f(x_0)+g(x_0))t+D(f+g)\;(f+g)(x_{0})\frac{t^2}2
+o(t^2).
\end{array}$$
Therefore from the Proposition we conclude that
$$x_{2t}=y_{2t}+[f,g](x_0)\frac{t^2}2+o(t^2).$$
Thus in (\ref{eqsys}) one switch between two admissible vector fields causes a deflection from the admissible trajectory provided by the average of the vector fields by a second order term proportional to their Lie bracket.
This is a simplified version of the well known Baker-Campbell-Hausdorff formula stopped at the second order.
In the above statements all remainders $o(t^2)$ become terms of the order of $t^3$ if the vector fields $f,g\in C^2$.
\end{rem}

\subsection{Second order Hamiltonians}

We now consider a $C^2$ function $u:\R^n\to\R$, $C^1$ vector fields $f,g:\R^n\to\R^n$ and define the second order hamiltonian
$$H_{f,g}(x)=\nabla(\nabla u\cdot f)\cdot g(x)=D^2u\;f\cdot g(x)+\nabla u\cdot Df\;g(x),$$
which corresponds to the standard notation for the usual hamiltonian $H_f(x)=-\nabla u\cdot f(x)$.

By standard Taylor expansions of functions in multiple variables we obtain second order estimates of the variation of functions along trajectories of (\ref{eqswitch}) as a consequence of Proposition \ref{prop1}.
\begin{prop}\label{prop2}
Let $t>0$ and $f,g:\R^n\to\R^n$ be $C^1$ vector fields. Let $u:\R^n\to\R$ be a function of class $C^2$. For the trajectory (\ref{eqswitch}) we have the following estimate
\begin{equation}\label{eqest}\begin{array}{ll}
u(x_{2t})-u(x_0)=\nabla u\cdot(f+g)(x_0)t\\+(H_{{f+g},{f+g}}(x_0)+\nabla u\cdot[f,g](x_0))\frac{t^2}2+o(t^2)\\
=\nabla u\cdot(f+g)(x_0)t\\
+(H_{f,f}(x_0)+H_{g,g}(x_0)+2H_{g,f}(x_0))\frac{t^2}2+o(t^2).
\end{array}\end{equation}
If in particular $f\equiv g$ then
$$u(x_{2t})-u(x_0)=2\nabla u\cdot f(x_0)t\\
+2H_{f,f}(x_0){t^2}+o(t^2).$$
\end{prop} 
\begin{proof}
From the standard Taylor estimate and Proposition \ref{prop1}
$$\begin{array}{ll}
u(x_{2t})-u(x_0)=\nabla u\cdot(f+g)(x_0)t\\
+\nabla u\cdot(D(f+g)\;(f+g)(x_0)+[f,g](x_0))\frac{t^2}2\\
+\frac12D^2u(f+g)\cdot(f+g)t^2+o(t^2)
\end{array}$$
from which the first equality follows. The second equality is a consequence of the simple observation that
$$\nabla u\cdot[f,g](x_0)=H_{g,f}(x_0)-H_{f,g}(x_0),$$
when $u\in C^2$.
\end{proof}
\begin{rem}
In the above, if the vector fields $f,g$ are at least of class $C^2$ and the function $u$ is at least of class $C^3$, then the remainders in Proposition \ref{prop2} are of the order $t^3$. Notice that in (\ref{eqest}) we obtain the variation of $u(x_\cdot)$ at the point $2t$ after a complete switch and not at every point of the interval $[0,2t]$.
\end{rem}

\section{Symmetric systems}

In this section we apply the previous one to symmetric systems, in particular for trajectories starting out at points where the vector fields are tangent. We then specialise at such points the principal part of (\ref{eqest}). To this end we first need to rewrite the second order hamiltonians.
\begin{lemma}\label{lemtech} Let $\sigma:\R^n\to\R^{n\times m}$ be of class $C^1$ whose columns are vector fields denoted by $\sigma_i$, $i=1,\dots,m$ and $u:\R^n\to\R$ be of class $C^2$. Let $a_1,a_2\in B_1(0)$ and $f=\sigma a_1,g=\sigma a_2$. Then:

\noindent
(i)  we can rewrite the second order hamiltonian as
$$H_{f,g}(x)=S(x)\;a_1\cdot a_2,\quad a_1,a_2\in B_1(0),
$$
where $S(x)={}^t\sigma(x)\;{}^tD(\nabla u\;\sigma)(x)$ and $S:\R^n\to \R^{m\times m}$ is continuous.

\noindent
(ii) We can express the product with the Lie bracket
$$[f,g]\cdot\nabla u(x)=2S^e(x)a_2\cdot a_1,
$$
where $S^e$ denotes the emisymmetric part of $S$.
In particular
$$S^e(x)=\left(\frac12[\sigma_j,\sigma_i]\cdot \nabla u(x)\right)_{i,j=1,\dots,m}$$
and $S$ is not symmetric at $x$ if and only if there is a Lie bracket among the vector fields $\sigma_i(x)$, $i=1,\dots,m$, which is not orthogonal to $\nabla u(x)$.

\noindent
(iii) The symmetric part of $S$ is 
$$\begin{array}{ll}
S^*(x)={}^t\sigma\;D^2u\;\sigma(x)\\
+\left(\frac12(D\sigma_j\;\sigma_i+D\sigma_i\;\sigma_j)\cdot\nabla u(x)\right)_{i,j=1,\dots,m}\end{array}$$
\end{lemma}
\begin{proof}
(i) It is just a simple computation
$$\begin{array}{ll}
H_{f,g}(x)=\nabla(\nabla u\; \sigma a_1)(x)\cdot \sigma(x) a_2=\\
{}^tD(\nabla u\;\sigma)(x)a_1\cdot\sigma(x) a_2={}^t\sigma(x)\;{}^tD(\nabla u\;\sigma)(x)\;a_1\cdot a_2.
\end{array}$$

\noindent
(ii) Again we compute
$$\begin{array}{ll}
[f,g]\cdot\nabla u(x)=H_{g,f}(x)-H_{f,g}(x)\\
=S(x)\;a_2\cdot a_1-S(x)\;a_1\cdot a_2=(S(x)-{}^tS(x))\;a_2\cdot a_1\\
=2S^e(x)a_2\cdot a_1.
\end{array}$$

\noindent
(iii) As easily seen in coordinates
$$\begin{array}{ll}
S(x)={}^t\sigma{}^tD(\nabla u \;\sigma(x))={}^t\sigma\;D^2u\;\sigma(x)\\
+(D\sigma_j\;\sigma_i\cdot\nabla u(x))_{i,j=1,\dots,m},
\end{array}$$
from which the symmetric part follows.
\end{proof}
We can rewrite the second order term in (\ref{eqest}) in two ways, that are convenient in different ways. Below we denote the matrix valued function $K:\R^n\to\R^{2m\times 2m}$,
\begin{equation}\label{eqkdef}
K(x)=\left(\begin{array}{cc}S^*(x)\quad&{}^tS(x)\\S(x)&S^*(x)
\end{array}\right).
\end{equation}
Notice that $K(x)$ is symmetric for any $x\in \R^n$. The proof of the following result is straightforward.
\begin{lemma}\label{lemtech2}
Let $\sigma:\R^n\to\R^{n\times m}$ be of class $C^1$ and $u:\R^n\to\R$ be of class $C^2$. Let $a_1,a_2\in B_1(0)$, then
$$\begin{array}{ll}
K(x)\left(\begin{array}{c}a_1\\a_2\end{array}\right)\cdot \left(\begin{array}{c}a_1\\a_2\end{array}\right)\\
=S(x)a_1\cdot a_1+S(x)a_2\cdot a_2+2S(x)a_1\cdot a_2\\
=S^*(x)(a_1+a_2)\cdot(a_1+a_2)+2S^e(x)a_1\cdot a_2.
\end{array}$$
\end{lemma}

We can now rephrase Proposition \ref{prop2} for symmetric systems.
\begin{prop}\label{prop3}
Let $t>0$ and $\sigma:\R^n\to\R^{n\times m}$ be of class $C^1$. Let $f=\sigma a_1$, $g=\sigma a_2$, for $a_1,a_2\in B_1(0)$ and $u:\R^n\to\R$ be a function of class $C^2$. The trajectory (\ref{eqswitch}) satisfies
\begin{equation}\label{eqestsym}\begin{array}{ll}
u(x_{2t})-u(x_0)=\nabla u\cdot\sigma(x_0)(a_1+a_2)t\\+K(x_0)\left(\begin{array}{c}a_1\\a_2\end{array}\right)\cdot \left(\begin{array}{c}a_1\\a_2\end{array}\right)\frac{t^2}2+o(t^2).
\end{array}\end{equation}
If in particular the vector fields $f,g$ are orthogonal to $\nabla u(x_0)$ at $x_0$, then
\begin{equation}\label{eqestsym2}\begin{array}{ll}
u(x_{2t})-u(x_0)=K(x_0)\left(\begin{array}{c}a_1\\a_2\end{array}\right)\cdot \left(\begin{array}{c}a_1\\a_2\end{array}\right)\frac{t^2}2+o(t^2).
\end{array}\end{equation}
\end{prop}
Proposition \ref{prop3} shows that when starting out at a point $x_0$ where all of the vector fields are tangent, in order to reach the negative side of the hypersurface $\{x:u(x)-u(x_0)=0\}$ one needs to know that a $2m\times 2m$ matrix is not positive semidefinite. Putting things together we prove our first main result.

\begin{theorem}\label{thm1}
Let $\sigma:\R^n\to\R^{n\times m}$ be of class $C^2$ and let $u:\R^n\to\R$ be a function of class $C^3$. Let $\bar x\in\R^n$ be a point such that $\nabla u\;\sigma(\bar x)=0$ and suppose that
\begin{equation}\label{eqpde}
\begin{array}{ll}
\max_{a_1,a_2\in B_1(0)}\{ -{\mbox Tr}(D^2u\; \sigma(a_1+a_2)\otimes\sigma(a_1+a_2) (\bar x))\\
-(D(\sigma (a_1+a_2))\;\sigma (a_1+a_2)+[\sigma a_1,\sigma a_2])\cdot\nabla u(\bar x)\}>0.
\end{array}\end{equation}
Then the target $\{x:u(x)\leq u(\bar x)\}$ is STLA by the system (\ref{eqsys}) at $\bar x$.
\end{theorem}
\begin{rem}
If (\ref{eqpde}) holds, then locally around $\bar x$ the inequality is preserved by continuity. Therefore (\ref{eqpde}) holds if and only if $u$ is a strict supersolution of the corresponding elliptic partial differential equation in a neighborhood of $\bar x$
\begin{equation}
\begin{array}{ll}
\max_{a_1,a_2\in B_1(0)}\{ -{\mbox Tr}(D^2u\; \sigma(a_1+a_2)\otimes\sigma(a_1+a_2) (x))\\
-[D(\sigma (a_1+a_2))\;\sigma (a_1+a_2)+[\sigma a_1,\sigma a_2])\cdot\nabla u(x)\}\\
\geq \rho>0,\quad x\in B_\rho(\bar x).
\end{array}\end{equation}
\end{rem}
\begin{proof}
Assume (\ref{eqpde}), clearly by previous discussion this amounts to
$$\max_{a_1,a_2\in B_1(0)}\left\{-K(\bar x)\left(\begin{array}{cc}a_1\\a_2\end{array}\right)\cdot 
\left(\begin{array}{cc}a_1\\a_2\end{array}\right)\right\}>0.$$
Then there are $a_1,a_2\in B_1(0)$ where the maximum is achieved. Let $f=\sigma a_1,\;g=\sigma a_2$, 
we follow the trajectory (\ref{eqswitch}) of the two vector fields such that $\nabla u\cdot f(\bar x)=0=\nabla u\cdot g(\bar x)$ starting out at $x_0=\bar x$. Then by (\ref{eqestsym2})
$$u(x_{2t})-u(\bar x)\leq-\rho\frac{t^2}2+Ct^3,$$
for $t$ positive and sufficiently small and for some constant $C$ estimating the remainder term. We want to keep $t\leq \rho/{4C}$ so that the right hand side remains strictly negative.

Fix any $0< t\leq \rho/{4C}$, we
now start the trajectory $x^1_\cdot$ in (\ref{eqswitch}) from a point $x_0=x^1$, $|x^1-\bar x|\leq\delta$. We obtain instead, by (\ref{eqestsym}),
\begin{equation}\label{eq2nd}
\begin{array}{ll}
u(x^1_{2t})-u(\bar x)\leq u(x^1_{2t})-u(x^1)+u(x^1)-u(\bar x)\\
\leq L_1\delta t-\rho \frac{t^2}2+Ct^3+L\delta\\
\leq \tilde L\delta-\rho \frac{t^2}2+Ct^3\leq  \tilde L\delta-\frac{t^2}2\frac\rho2,
\end{array}\end{equation}
where $L$ is a local Lipschitz constant for $u$, $L_1$ is a local Lipschitz constant for the product $\nabla u(x)\cdot (f(x)+g(x))$ which is zero at $\bar x$ and
$\tilde L=L+(L_1\rho)/(4C)$. The estimate is uniform on the starting point $x^1$ as well as the radius $\delta$ at least locally. If moreover we select, for $t$ sufficiently small,
$$\delta= { t}^2\frac{\rho}{4{\tilde L}}$$
then the right hand side of (\ref{eq2nd}) is zero. Therefore the trajectory (\ref{eqswitch}) starting at any $x^1$ will reach the target $\{x:u(x)\leq u(\bar x)\}$ earlier than $\bar t$. In particular we can estimate the minimum time to reach the target from any point $x^1$ in the neighborhood of $\bar x$ as
$$T(x^1)\leq 2\sqrt{\frac{\tilde L\delta}{\rho}},\quad |x^1-\bar x|\leq\delta,$$
namely with the square root of the distance from the center on the ball on the target.
Hence the system is STLA at $\bar x$.
\end{proof}

\begin{rem}
With a similar argument of the previous proof, if we know that 
\begin{equation}\label{eq1storder}
\nabla u\cdot f(x)\leq -\rho<0
\end{equation}
in the neighborhood of $\bar x$, and  $(x_t)_{t\geq 0}$ is the trajectory of the vector field $f$, we obtain an estimate of the form
$$u(x_t)-u(\bar x)\leq -\rho t+Ct^2,$$
for $t$ sufficiently small. Here the leading negative term has a first order power in $t$. If now $|x^1-\bar x|\leq\delta$ and we follow the trajectory of $f$ starting at $x^1$, call it $(x^1_t)_{t\geq0}$, then the estimate becomes
\begin{equation}\label{eq1st}
u(x^1_t)-u(\bar x)\leq -\rho t+Ct^2+L\delta,\end{equation}
where $L$ is a local Lipschitz constant for $u$. It follows from here that the minimum time to reach the target from $x_1$ can be estimated as
$$T(x_1)\leq\frac{2L}\rho\delta,$$
therefore with the distance from the center of the ball on the target. Thus the target is STLA by the system at $\bar x$.
\end{rem}
\begin{rem}
If the estimate (\ref{eq1storder}) holds at every point of the target, the one proves that the minimum time function is locally Lipschitz continuous in its domain by adapting the argument in \cite{so}.
If instead at every point of the boundary of the target it holds either (\ref{eq1storder}) or (\ref{eqpde}) then
from Theorem \ref{thm1} one can prove that the minimum time function is locally 1/2-H\"older continuous.
This fact outlines the difference between a first and a second order conditions.
\end{rem}

\section{Analysis of $K$}

In this section we will compute the minimum of the function ($K$ is as in (\ref{eqkdef}))
\begin{equation}\label{eqmin}
h(a_1,a_2)=K(\bar x)\left(\begin{array}{cc}a_1\\ a_2\end{array}\right)\cdot \left(\begin{array}{cc}a_1\\ a_2\end{array}\right),\quad |a_1|,|a_2|\leq 1,\end{equation}
and then characterize when it is strictly negative through the properties of $S(\bar x)$. This will require some linear algebra. A consequence will be that if at $\bar x$ all vector fields of the system are orthogonal to $\nabla u(\bar x)$, then (\ref{eqsys}) is STLA at $\bar x$. This will be our second main result.
The coordinates of the point where the minimun is attained will then provide the controls for two vector fields that allow us to cross the hypersurface with maximal rate of decrease. 
\begin{prop}
The minimum of (\ref{eqmin}) is non positive. If (\ref{eqmin}) attains a negative minimum, then it is reached at an eigenvector $v=(a_1,a_2)$ of $K(\bar x)$ with minimal eigenvalue $\lambda$ and we have $|a_1|=|a_2|=1$, $h(v)=2\lambda$.
\end{prop}
\begin{proof}
It is clear that the minimum of $h$ in (\ref{eqmin}) is nonpositive, as by definition of $K$ in (\ref{eqkdef}), $h(a,-a)=0$ for all $a\in B_1(0)$, thus 0 is an eigenvalue of $K$. We will not write the dependence on $\bar x$ below. Also notice that the minimum of $h$ in the ball $B_{\sqrt{2}}(0)$, which contains the domain of $h$, is attained at an eigenvalue of norm $\sqrt{2}$ of the minimal eigenvalue.
We now show that if $(a_1,a_2)$ is an eigenvector of $K$ with non zero eigenvalue, then $|a_1|=|a_2|$. Therefore the minimum in (\ref{eqmin}) is also attained at an eigenvector of $K$ with norm $\sqrt{2}$ with minimal eigenvalue.

Let $(a_1,a_2)$ be an eigenvector of $K$ with $\lambda$ as an eigenvalue. Then it satisfies
\begin{equation}\label{eqsystem}
\left\{\begin{array}{ll}
S^*a_1+{}^tSa_2=\lambda a_1,\\
S^*a_2+Sa_1=\lambda a_2.
\end{array}\right.
\end{equation}
Multiply the first equation of (\ref{eqsystem}) by $a_2$ and the second by $a_1$. We obtain
\begin{equation}\label{eqsystem1}\left\{\begin{array}{ll}
S^*a_1\cdot a_2+{}^tSa_2\cdot a_2=\lambda a_1\cdot a_2,\\
S^*a_2\cdot a_1+Sa_1\cdot a_1=\lambda a_1\cdot a_2,
\end{array}\right.\end{equation}
and then
\begin{equation}\label{eqss}
Sa_1\cdot a_1=Sa_2\cdot a_2.\end{equation}
Now restart from (\ref{eqsystem}) and multiply the first equation by $a_1$ and the second by $a_2$. We obtain
\begin{equation}\label{eqsystem2}\left\{\begin{array}{ll}
S^*a_1\cdot a_1+{}^tSa_2\cdot a_1=\lambda |a_1|^2,\\
S^*a_2\cdot a_2+Sa_1\cdot a_2=\lambda |a_2|^2,
\end{array}\right.\end{equation}
and therefore by (\ref{eqss})
$$\lambda(|a_1|^2-|a_2|^2)=0,$$
which gives us the conclusion.
\end{proof}
In the next result we characterise when the minimum in (\ref{eqmin}) is negative by properties of $S$.
\begin{theorem}
Let $\sigma:\R^n\to\R^{n\times m}$ be of class $C^2$ and let $u:\R^n\to\R$ be a function of class $C^3$. Let $\bar x\in\R^n$ be a point such that $\nabla u\;\sigma(\bar x)=0$. 
Assume that either $S(\bar x)$ is not symmetric, or if it is, it has at least one negative eigenvalue. Then there are $a_1, a_2\in B_1(0)$ such that
$$K(x_0)\left(\begin{array}{cc}a_1\\a_2\end{array}\right)\cdot \left(\begin{array}{cc}a_1\\a_2\end{array}\right)<0.$$
In particular $K$ has a negative eigenvalue and the system (\ref{eqsys}) is STLA at $\bar x$.
\end{theorem}
\begin{proof}
1. We suppose first that $S(\bar x)$ is symmetric. In the rest of the proof we will drop the dependence of the matrices on $\bar x$. Then by Lemma \ref{lemtech2}
$$K\left(\begin{array}{cc}a_1\\a_2\end{array}\right)\cdot \left(\begin{array}{cc}a_1\\a_2\end{array}\right)=S(a_1+a_2)\cdot (a_1+a_2).$$
It is therefore clear that when $S$ is positive semidefinite then the minimum of (\ref{eqmin}) is zero. Otherwise it has an unit eigenvector $v$ with negative eigenvalue, and we can choose $a_1=a_2=v$ to reach our goal.

2. Suppose now that $S$ is not symmetric. In particular $S$ is not the null matrix. Consider the positive semidefinite matrix ${}^tSS$, it will have at least one positive eigenvalue $\lambda^2$ with corresponding unit eigenvector $a_1$. Thus 
$${}^tSSa_1=\lambda^2 a_1$$ 
and then
$$|Sa_1|^2=Sa_1\cdot Sa_1={}^tSSa_1\cdot a_1=\lambda^2 a_1\cdot a_1=\lambda^2,$$
so that $\lambda=|Sa_1|>0$. Just notice that if $\bar a$ is eigenvector of ${}^tSS$ with null eigenvalue, then the same argument shows that $S\bar a=0$.
Let 
$$a_2=-\frac{S(x_0)a_1}\lambda.$$
Now we obtain
$${}^tSa_2=-\lambda a_1,\quad Sa_1\cdot a_2=-\lambda$$
$$Sa_2\cdot a_2=-\lambda a_2\cdot a_1=Sa_1\cdot a_1.
$$
Thus we conclude that
$$K
\left(\begin{array}{cc}a_1\\ a_2\end{array}\right)\cdot \left(\begin{array}{cc}a_1\\ a_2\end{array}\right)=-2\lambda(1+a_1\cdot a_2).$$
We reach our conclusion provided $a_1\neq -a_2$. 
Let us analyse this critical case. By definition it then follows
$$Sa_1=\lambda a_1,\quad {}^tSa_1=\lambda a_1.
$$
Therefore if this critical case happens for all eigenvectors of ${}^tSS$ with positive eigenvalues, and we consider an orthonormal basis of eigenvectors of ${}^tSS$, this is also a family of eigenvectors for $S$ which can then be diagonalized by an orthogonal matrix and is thus symmectric, which was supposed not to be the case.

The conclusion now follows from Theorem \ref{thm1}.
\end{proof}

\section{Examples}

\begin{example}
In this example we want to show that our condition for second order attainability can be satisfied by a single vector field. Consider the system
\begin{equation}
\left\{
\begin{array}{ll}
\dot x_t=-ay_t, \\
\dot y_t=ax_t  \\
(x_0,y_0)\in\R^2.  
\end{array}
\right.
\end{equation}
Here $a\in [-1,1]$. Let $u(x,y)=y-1$, $\sigma(x,y)={}^t(-y,x)$. Around $(x_0,y_0)=(0,1)$ we want to reach the target $\{x:y\leq1\}$. Since $\nabla u(x,y)=(0,1)$ then $\nabla u(0,1)\sigma(0,1)=0$ and first order conditions do not apply. Instead we compute $S(x,y)=-y$, notice that it is scalar (symmetric) and negative for $y=1$. Indeed in this case
\begin{equation}\label{eqkk}
K=\left(\begin{array}{cc} -1\quad&-1\\-1&-1
\end{array}\right),\end{equation}
and $K$ has $(1,1)$ as an eigenvector of $-2$ as eigenvalue. Therefore the target is small time locally attainable at $(0,1)$. 
\end{example}


\begin{example}
(This example comes from \cite{ma2}). In $\R^2$, take $\sigma={}^t(0,1)$ and $u(x,y)=\frac{1-x^2+y^2}2$ so there is a unique vector field which is constant. However $\nabla u\cdot\sigma(x,y)=-y$, therefore we have first order attainability of the sublevel sets of $u$ unless $y=0$. At every point, in particular at $(1,0)$ we have $S=-1<0$ so there is second order attainability of $\{x:u(x,y)\leq0\}=\R^2\backslash B_1((0,0))$. Matrix $K$ is as in (\ref{eqkk}).
\end{example}

\begin{example}
Consider in $\R^2$ the system where
$$\sigma(x,y)=\left(
\begin{array}{cc}
y\quad & 0  \\
 0& 1
\end{array}
\right),\quad u(x,y)=\frac{x^2+y^2}2.
$$
Then $\nabla u\;\sigma (x,y)={}^t(xy\quad y)$ which vanishes at points where $y=0$. We impose $x\neq0$ otherwise the gradient vanishes and look for second order conditions. Computing $S$ at such points we get
$$S(x,0)=\left(
\begin{array}{cc}
0\quad & 0  \\
 x& 1
\end{array}\right).$$
Since $S$ is not symmetric we know that the system satisfies an attainability condition of second order.
\end{example}

\begin{example}(Heisenberg system)
In $\R^3$ consider the system where
$$\sigma(x,y,z)=\left(
\begin{array}{cc}
1\quad & 0  \\
 0& 1\\
 y&-x
\end{array}
\right),\quad u(x,y)=\frac{x^2+y^2+z^2}2.
$$
Then $\nabla u\;\sigma (x,y)={}^t(x+yz\quad y-xz)$ which vanishes at points where $x=y=0$, and we select $z\neq0$ because otherwise the gradient of $u$ vanishes. Computing $S$ at such points we get
$$S(0,0,z)=\left(
\begin{array}{cc}
1\quad & -z  \\
 z& 1
\end{array}\right).$$
which again is not symmetric and we know that the system satisfies an attainability condition of second order of the sublevel sets of $u$. In this case we computed the minimal eigenvalue at $z=1$ which has multiplicity 2
$$K(0,0,1)=\left(
\begin{array}{cccc}
1\quad & 0  \quad &1\quad &1\\
 0& 1&-1&1\\
1&-1&1&0\\
1&1&0&1
\end{array}
\right),
$$
and $\lambda_{min}=1-\sqrt{2}<0$. Eigenvectors providing the highest decrease rate of $u$ are $(-\frac{\sqrt{2}}2,\frac{\sqrt{2}}2,1,0),(-\frac{\sqrt{2}}2,-\frac{\sqrt{2}}2,0,1)$ and the vector space generated by them. Each of the two pairs of coordinates, e.g. $(-\frac{\sqrt{2}}2,\frac{\sqrt{2}}2),(1,0)$, give us controls to determine the two vector fields that we need to use to achieve attainability of teh sublevel sets of $u$ with maximal rate.
\end{example}

\begin{example}(Dubin's system)
In $\R^3$ take the system where
$$\sigma(x,y,z)=\left(
\begin{array}{cc}
\cos{z}\quad & 0  \\
 \sin{z}& 0\\
 0&1
\end{array}
\right),\quad u(x,y)=\frac{x^2+y^2+z^2}2.
$$
Therefore $\nabla u\;\sigma (x,y,z)={}^t(x\cos{z}+y\sin{z}\quad z)$ which vanishes at points where $x=z=0$, and we add $y\neq0$ because otherwise the gradient of $u$ vanishes. Computing $S$ at such points we get
$$S(0,y,0)=\left(
\begin{array}{cc}
1\quad & 0  \\
 y& 1
\end{array}\right).$$
which is not symmetric. Hence the sublevel sets of $u$ are STLA around $(0,y,0)$, $y\neq0$.
\end{example}

\section{CONCLUSIONS AND FUTURE WORKS}

We proposed a new sufficient second order condition for STLA of a smooth target by checking if a symmetric matrix has a negative eigenvalue. This condition is triggered if either there is transversal Lie bracket to the target or if the geometry of the target  allows attainability with a single vector field. The eigenvector contains the controls that we need to use to reach the target with at most one switch.

Possible future developments of this work point in different directions. The way our approach reads for symmetric systems is clean and simple. It is not so for general nonlinear systems in particular affine systems with nontrivial drift. We will explore this in detail elsewhere.
One may want to design a control globally to steer the system to a target in finite time or asymptotically. We will also cope with that elsewhere.

\section{ACKNOWLEDGMENTS}

The author wishes to thank Mauro Costantini for useful discussions on algebra of matrices.

\end{document}